\documentclass[fontsize=12pt]{article}

\usepackage{graphicx}
\usepackage{amssymb}
\usepackage{amsmath}
\usepackage{amsthm}

\usepackage{algpseudocode}
\usepackage{algorithm}

\newcommand{\bO}{{\mathcal{O}}}

\newtheorem{theorem}{Theorem}[section]
\newtheorem{definition}{Definition}[theorem]
\newtheorem{corollary}{Corollary}[theorem]
\newtheorem{lemma}[theorem]{Lemma}
\newtheorem{remark}{Remark}[theorem]
\newtheorem{assumption}{Assumption}[section]
\title{Solving Non-Convex Non-Concave Min-Max Games Under Polyak-{\L}ojasiewicz Condition}
\author{Maziar Sanjabi, Meisam Razaviyayn, Jason D. Lee\\
University of Southern California}
\date{}                                           

\begin{document}
\maketitle

\begin{abstract}
In this short note, we consider the problem of solving a min-max zero-sum game. This problem has been extensively studied in the convex-concave regime where the global solution can be computed efficiently.  Recently, there have been developments for finding the first order stationary points of the game when one of the player's objective is  concave or (weakly) concave. This work focuses on the non-convex non-concave regime where the objective of one of the players satisfies Polyak-{\L}ojasiewicz (PL) Condition. For such a game, we show that a simple multi-step gradient descent-ascent algorithm finds an $\varepsilon$--first order stationary point of the problem in $\widetilde{\bO}(\varepsilon^{-2})$ iterations.

\end{abstract}

\section{Problem Formulation}
Consider the  min-max optimization problem
\begin{align}
\min_\theta\max_\alpha \;\;f(\theta,\alpha).\label{eq: game1}
\end{align}
This optimization problem can be viewed as a zero-sum game between two players where the goal of the first player is to maximize the objective function $f(\cdot,\cdot)$, while the other player's objective is to minimize the objective function. 
Our goal is to develop an algorithm for finding a first order stationary point of the above optimization problem in the non-convex non-concave regime. To this end, let us first define some initial concepts to simplify the presentation of ideas:
\begin{definition}[Stationarity]\label{def: stationarity}
A point $(\theta,\alpha)$ is said to be a first order stationary solution of the game \eqref{eq: game1}, if 
\begin{align}
\nabla_\theta f(\theta, \alpha)=0 \quad {\rm   and  } \quad\nabla_\alpha f(\theta, \alpha)=0.
\end{align}
\end{definition}
Note that Definition~\ref{def: stationarity} is a necessary condition for a (local) Nash equilibrium and is a special case of the Quasi Nash Equilibrium (QNE) defined in~\cite{pang2016unified}.

Given any positive $\varepsilon$, we can also define an approximate stationary solution $(\theta_\varepsilon, \alpha_\varepsilon)$ as:
\begin{definition}[Approximate Stationarity]\label{def: stationarity}
For any $\varepsilon>0$, a point $(\theta_\varepsilon,\alpha_\varepsilon)$ is said to be an $\varepsilon$--stationary solution of the game \eqref{eq: game1}, if 
\begin{align}
\|\nabla_\theta f(\theta, \alpha)\|\leq\varepsilon\quad {\rm and} \quad \|\nabla_\alpha f(\theta, \alpha)\|\leq \varepsilon.
\end{align}
\end{definition}


\vspace{0.4cm}
Throughout this work, we assume that the function $f(\cdot,\cdot)$ is smooth and well-behaved; more precisely, we will make the following assumption:
\begin{assumption}
\label{assumption: LipSmooth}
The function $f$ is second order continuously differentiable in both $\theta$ and $\alpha$; moreover, its gradients with respect to $\alpha$ and $\theta$ is Lipchitz continuous. More precisely, there exists constants $L_{11}$, $L_{22}$ and $L_{12}$ such that for every $\alpha,\alpha_1,\alpha_2,\theta,\theta_1,\theta_2$, we have
\begin{align}
&\|\nabla_\theta f(\theta_1,\alpha)-\nabla_\theta f(\theta_2,\alpha)\|\leq L_{11}\|\theta_1-\theta_2\|\nonumber\\
&\|\nabla_\alpha f(\theta,\alpha_1)-\nabla_\alpha f(\theta,\alpha_2)\|\leq L_{22}\|\alpha_1-\alpha_2\|\nonumber\\
&\|\nabla_\alpha f(\theta_1,\alpha)-\nabla_\alpha f(\theta_2,\alpha)\|\leq L_{12}\|\theta_1-\theta_2\|\nonumber\\
&\|\nabla_\theta f(\theta,\alpha_1)-\nabla_\theta f(\theta,\alpha_2)\|\leq L_{12}\|\alpha_1-\alpha_2\|\nonumber
\end{align}  
\end{assumption}

\vspace{0.4cm}

\section{A Simple Iteration Complexity Lower Bound}
\label{sec:lowerbound}
Our goal is to develop an ``efficient" algorithm  for finding an approximate stationary point of \eqref{eq: game1}. A standard approach to measure the efficiency of a given algorithm is based on the number of gradient evaluations of the algorithm.  For our min-max problem, one simple approach to obtain a lower bound on the number of required gradient evaluations to  find an $\varepsilon$--stationary is based on the standard lower complexity bound of solving non-convex optimizations \cite{nesterov2013introductory}. In particular, we consider the case where  $f(\theta,\alpha_1) = f(\theta,\alpha_2), \;\forall \alpha_1,\alpha_2$, i.e., $f(\theta,\alpha)$ is independent of the choice of $\alpha$. In this case the known lower bounds for solving \eqref{eq: game1} coincides with the lower bound on finding the stationary solutions of the general non-convex function $w(\theta) \triangleq f(\theta,\alpha )$. It is well known that for such a problem, finding an $\varepsilon$--stationary solution requires $\bO(\varepsilon^{-2})$ gradient evaluations  which  can be achieved by simple gradient descent algorithm \cite{carmon2017lower,nesterov2013introductory}. In the next section, we will  show that this lower bound can also be achieved (up to logarithmic factors) for a class of non-convex non-concave min-max game problems~\eqref{eq: game1}.

\section{Finding Approximate Stationary Points in Min-Max Games}
Notice that solving the optimization problem~\eqref{eq: game1} is equivalent to solving 
\begin{align}
\min_\theta \quad  g(\theta), \label{eq: game_min}
\end{align}
where 
\begin{align}
g(\theta) \triangleq \max_{\alpha} f(\theta,\alpha). \label{eq:gEval}
\end{align}
Thus, in order to solve~\eqref{eq: game1}, it seems natural to apply first order algorithms to the optimization problem~\eqref{eq: game_min}  directly (if possible). However, in the new optimization problem~\eqref{eq: game_min}, even evaluating the objective function $g(\theta)$ requires solving another optimization problem, i.e., maximizing $f(\alpha, \theta)$ over~$\alpha$. Therefore, to apply this idea,  we are going to assume that solving \eqref{eq:gEval} is ``easy" and satisfies Polyak-{\L}ojasiewicz condition. To clarify our assumption, let us first define the Polyak-{\L}ojasiewicz (PL) condition:

%
%
%
\begin{definition}[Polyak-{\L}ojasiewicz Condition]\label{def: PL}
A differentiable function $h(x)$ with the minimum value $h^* = \min_x\; h(x)$ is said to be $\mu$-Polyak-{\L}ojasiewicz ($\mu$-PL) if
\begin{align}
\frac{1}{2}\|\nabla h(x)\|^2\geq \mu (h(x)-h^*),\quad \forall x
\end{align}
\end{definition}
It is well-known that PL condition does not necessarily imply  convexity of the function and it may hold even for some non-convex functions~\cite{PL_karimi_2016}. Based on this PL condition, we define the concept of PL-game as follows:

\begin{definition}[PL-Game]\label{def: PLGame}
We say that the min-max game  \eqref{eq: game1} is a PL-Game if there exists a constant $\mu>0$ such that the  function $h_\theta(\alpha) \triangleq -f(\theta, \alpha)$ is $\mu$-PL for any fixed value of~$\theta$.
\end{definition}

For the class of PL-Games, one can show that the gradient of the function $g(\cdot)$ in \eqref{eq:gEval} can be ``approximately" computed at any given point $\theta$. Thus, we can use such a gradient for finding a stationary point of  \eqref{eq: game_min}.  In what follows, we first outline an algorithm which utilizes this idea to compute a stationary point of   \eqref{eq: game1}. We then rigorously analyze the proposed algorithm.

\begin{algorithm}
  \caption{Multi-step Gradient  Descent Ascent} 
  \label{alg: alg_grad}

  \begin{algorithmic}
  	\State INPUT:  $K$, $T$, $\eta_1$, $\eta_2$, $\alpha_0$ and $\theta_0$
	\For  {$t=0, \cdots, T-1$}
	\State Set $\alpha_0(\theta_t) = \alpha_t$
	\For {$k=0, \cdots, K-1$}
	\State Set $\alpha_{k+1}(\theta_t) = \alpha_{k}(\theta_t) + \eta_1 \nabla_{\alpha}f(\theta_t, \alpha_{k}(\theta_t))$
	\EndFor
	\State Set $\alpha_{t+1} = \alpha_{K}(\theta_t)$
	\State Set $\theta_{t+1} = \theta_t-\eta_2 \nabla_\theta f(\theta_t, \alpha_{t+1})$
	\EndFor
	\State Return $(\theta_t, \alpha_{t+1})$~for $t=0,\cdots,T-1$.
  \end{algorithmic}
\end{algorithm}

The inner loop in Algorithm~\ref{alg: alg_grad} tries to solve the optimization problem~\eqref{eq:gEval} for a given fixed value of $\theta = \theta_t$. The result of this optimization problem can be used to compute the gradient of the function $g(\theta)$. In other words,  one can show a ``Danskin-type" result \cite{danskin_result_1995}  and imply that $\nabla_\theta f(\theta_t,\alpha_{t+1}) \approx \nabla g(\theta_t)$. More precisely,  we can show the following lemma:

\begin{lemma}\label{lemma: conv_alpha_main}
Define $\kappa = \frac{L_{22}}{\mu}\geq 1$ and $\rho = 1-\frac{1}{\kappa}<1$ and assume $g(\theta_t) - f(\theta_t, \alpha_0(\theta_t))<\Delta$, then for any $\varepsilon>0$ if we choose $K$ large enough such that
\begin{align}
16 \frac{\max(L_{22}, L_{12})^2 \Delta}{\mu}\rho^K \leq \varepsilon^2,
\end{align}
or in other words
\begin{align}
K\geq N_1({\varepsilon}) = \frac{2\log(1/\varepsilon) + \log(16\bar{L}^2\Delta/\mu)}{\log(1/\rho)}, 
\end{align}
where $\bar{L} = \max(L_{12},L_{22})$, then 
\begin{align}
\|\nabla_\theta f(\theta_t,\alpha_K(\theta_t))-\nabla g(\theta_t)\|\leq \frac{\varepsilon}{4}\quad {\rm and }\quad \|\nabla_\alpha f(\theta_t,\alpha_K(\theta_t))\|\leq \varepsilon,
\end{align}
where in $\nabla_\theta f(\theta_t,\alpha_K(\theta_t)),$ the gradient is computed with respect to the first argument of the function $f(\cdot,\cdot) $ only; and in $\nabla_\alpha f(\theta_t,\alpha_K(\theta_t)),$ the gradient is computed with respect to the second argument of the function $f(\cdot,\cdot) $ only.
\end{lemma}

%
%
%
%
%

The above lemma implies that Algorithm~\ref{alg: alg_grad} behaves similar to the gradient descent algorithm applied to the problem~\eqref{eq: game_min}. Thus, its convergence to a stationary point can be established similar to the gradient descent algorithm. More precisely, we can show the following theorem.

\begin{theorem}\label{thm:main}
Define $L = L_{11}+ \frac{L_{12}^2}{\mu}$, $\bar{L} = \max(L_{12},L_{22})$, and $\Delta_g = g(\theta_0)-g^*$, where $g^* \triangleq \min_{\theta} \;g(\theta)$ is the optimal value of $g$. Morover, assume that $g(\theta_t)-f(\theta_t, \alpha_0(\theta_t))\leq \Delta$ for all iterations in Algorithm \ref{alg: alg_grad}. If we apply Algorithm \ref{alg: alg_grad} with $\eta_1 = \frac{1}{L_{22}}$, $\eta_2 = \frac{1}{L}$, $K \geq N_1(\varepsilon)= \frac{2\log(1/\varepsilon) + \log(16\bar{L}^2\Delta/\mu)}{\log(1/\rho)}$, and $T\geq \frac{18L\Delta_g}{\varepsilon^2}$, then there exists an iteration $t\in \{0,\cdots, T\}$ such that $(\theta_t,\alpha_{t+1})$ is an $\varepsilon$--stationary point of \eqref{eq: game1}.
\end{theorem}
It is worth noticing that the condition on the initial solutions is not that restrictive and as far as the optimal solution to $g$ is bounded and the iterates remain bounded would be easily satisfied throughout the process. This is mainly because the step-size would be small and the optimal $\alpha$ does not change that much from each iteration to the next as the changes in $\theta$ would be small; see Lemma \ref{lemma: stability} in Appendix.
\begin{corollary}
The above theorem implies that in order to obtain an $\varepsilon$--stationary solution of the game~\eqref{eq: game1},   $O(\varepsilon^{-2})$  evaluations of the gradient of the objective with respect to $\theta$ is needed; similarly,  $O(\varepsilon^{-2}\log(\frac{1}{\varepsilon}))$ gradients with respect to the $\alpha$  is required. If the two oracles have the same complexity, the overall complexity of the method would be $O(\varepsilon^{-2}\log(\frac{1}{\varepsilon}))$ which is only a logarithmic factor away from the lower bound in section~\ref{sec:lowerbound}.
\end{corollary}

\begin{remark}
In \cite[Theorem 4.2]{gan_sanjabi_18}, we have shown a similar result  for the case when $f(\theta,\alpha)$ is strongly concave in $\alpha$. Hence,  Theorem~\ref{thm:main} can be viewed as an extension of  \cite[Theorem 4.2]{gan_sanjabi_18}. Similar to \cite[Theorem 4.2]{gan_sanjabi_18}, one can easily extend the result of Theorem~\ref{thm:main} to the stochastic setting as well where we replace the gradients with stochastic gradients.
\end{remark}

\bibliographystyle{abbrv}
\bibliography{ref}
\appendix

\section{Proofs}
In this section, we prove the main result of this note, i.e., Theorem~\ref{thm:main}. To state the proof,  we need some intermediate lemmas and some preliminary definitions:
\begin{definition} \cite{QG_anitescu2000}
A function $h(x)$ is said to satisfy the Quadratic Growth (QG) condition with constant $\gamma>0$ if 
\begin{align}
    h(x) - h^* \geq \frac{\gamma}{2}{\rm dist}(x)^2, \quad \forall x,
\end{align}
where $h^*$ is the minimum value of the function, and ${\rm dist}(x)$ is the  distance of the point $x$ to the optimal solution set.
\end{definition}
The following lemma shows that PL implies QG \cite{PL_karimi_2016}.
\begin{lemma}[Corollary of Theorem 2 in \cite{PL_karimi_2016}]\label{lemma: QG}
If function $f$ is PL with constant $\mu$, then $f$ satisfies quadratic growth condition with constant $\gamma =4\mu$.
\end{lemma}
The next Lemma shows the stability of $\arg\max_\alpha f(\theta,\alpha)$ with respect to $\theta$ under PL condition.
\begin{lemma} \label{lemma: stability}
Assume that $\{h_{\theta}(\alpha) =-f(\theta, \alpha) ~|~\theta\}$ is a class of $\mu$-PL functions in $\alpha$. Define $A(\theta) = \arg\max_{\alpha} f(\theta,\alpha)$ and assume $A(\theta)$ is closed. Then for any $\theta_1$, $\theta_2$ and $\alpha_1\in A(\theta_1)$, there exists an $\alpha_2\in A(\theta_2)$ such that 
\begin{align}
    \|\alpha_1-\alpha_2\|\leq \frac{L_{12}}{2\mu}\|\theta_1-\theta_2\|
\end{align}
\end{lemma}
\begin{proof}
Based on the Lipchitzness of the gradients, we have that $\|\nabla_\alpha f(\theta_2,\alpha_1)\|\leq L_{12}\|\theta_1 - \theta_2\|$. Thus, based on the PL condition, we know that
\begin{align}
    g(\theta_2) - h_{\theta_2}(\alpha_1)\leq \frac{L_{12}^2}{2\mu}\|\theta_1 - \theta_2\|^2.
\end{align}
Now we use the result of Lemma \ref{lemma: QG} to show that there exists $\alpha_2\in A(\theta_2)$ such that
\begin{align}
    2\mu\|\alpha_1-\alpha_2\|^2\leq \frac{L_{12}^2}{2\mu}\|\theta_1 - \theta_2\|^2 
\end{align}
re-arranging the terms, we get the desired result that $\|\alpha_1-\alpha_2\|\leq \frac{L_{12}}{2\mu}\|\theta_1-\theta_2\|$.
\end{proof}
Now we  use this result to prove that the function $g(\theta) = \max_\alpha f(\theta, \alpha)$ is in fact smooth.
\begin{lemma}\label{lemma: smoothness}
Under the assumptions of Lemma \ref{lemma: stability}, function $g$ is $L$-Lipchitiz smooth with $L = L_{11} + \frac{L_{12}^2}{2\mu}$.
\end{lemma}
\begin{proof}
Let us compute the directional derivative at any point $\theta$ and direction $d$:
\begin{align}
    g'(\theta; d) = \lim_{t\rightarrow 0^+} \frac{g(\theta + t d) - g(\theta)}{t}.
\end{align}
Let us now fix one $\alpha \in A(\theta)$. Based on Lemma \ref{lemma: stability}, for any $t$, we have an $\alpha(t)$, such that $\|\alpha(t)-\alpha\|\leq \frac{L_{12}}{2\mu}t\|d\|$. Based on this inequality, we can write the Taylor expansion for
\begin{align}
    g(\theta + td) - g(\theta) &= f(\theta+td , \alpha(t)) - f(\theta, \alpha) \nonumber\\
   &= t\nabla_\theta f(\theta,\alpha)^T d + \underbrace{\nabla_\alpha f(\theta,\alpha)^T}_{0}(\alpha(t)-\alpha) + \bO(t^2).
\end{align}
Thus, we can easily conclude that 
\begin{align}
g'(\theta; d) = \lim_{t\rightarrow 0^+} \frac{g(\theta + t d) - g(\theta)}{t} = \nabla_\theta f(\theta,\alpha)^T d.
\end{align} 
Note that this relationship holds for any $d$. Thus, $\nabla g(\theta) = \nabla_\theta f(\theta,\alpha)$ for $\alpha\in A(\theta)$. Interestingly, the directional derivative does not depend on the choice of $\alpha$. This means that $\nabla_\theta f(\theta,\alpha_1) = \nabla_\theta f(\theta,\alpha_2)$ for any $\alpha_1$ and $\alpha_2$ in $A(\theta)$.
\end{proof}
Finally, the following lemma would be useful in the proof of  Theorem \ref{thm:main}. 

\begin{lemma}[See Theorem 5 in \cite{PL_karimi_2016}]\label{lemma: karimi_PL_conv}
Assume $h(x)$ is $\mu$-PL and $L$-smooth. Then, by applying gradient descent with step-size $1/L$ from point $x_0$ for $K$ iterations we get an $x_K$ such that
\begin{align}
h(x) - h^*\leq \Big(1-\frac{\mu}{L}\Big)^K(h(x_0)-h^*),
\end{align}
where $h^* = \min_x h(x)$.


\end{lemma}

We are now ready to prove the convergence of Algorithm \ref{alg: alg_grad}.

\subsection{Proof of Theorem \ref{thm:main}}

\begin{proof}
Our proof strategy is to show that by performing enough steps of gradient ascent on $\alpha$, before updating $\theta$ in Algorithm \ref{alg: alg_grad}, we can mimic the behavior of  the gradient descent algorithm  applied to $g$. The following lemma formalizes such an statement.
\begin{lemma}\label{lemma: conv_alpha}
Define $\kappa = \frac{L_{22}}{\mu}\geq 1$ and $\rho = 1-\frac{1}{\kappa}<1$ and assume $g(\theta_t) - f(\theta_t, \alpha_0(\theta_t))<\Delta$, then for any $\varepsilon>0$ if we choose $K$ large enough such that
\begin{align}
16 \frac{\max(L_{22}, L_{12})^2 \Delta}{\mu}\rho^K \leq \varepsilon^2,
\end{align}
or in other words
\begin{align}
K\geq N_1({\varepsilon}) = \frac{2\log(1/\varepsilon) + \log(16\bar{L}^2\Delta/\mu)}{\log(1/\rho)}, 
\end{align}
where $\bar{L} = \max(L_{12},L_{22})$, then 
\begin{align}
\|\nabla_\theta f(\theta_t,\alpha_K(\theta_t))-\nabla g(\theta_t)\|\leq \frac{\varepsilon}{4}\quad {\rm and }\quad \|\nabla_\alpha f(\theta_t,\alpha_K(\theta_t))\|\leq \varepsilon,
\end{align}
where in $\nabla_\theta f(\theta_t,\alpha_K(\theta_t)),$ the gradient is computed with respect to the first argument of the function $f(\cdot,\cdot) $ only; and in $\nabla_\alpha f(\theta_t,\alpha_K(\theta_t)),$ the gradient is computed with respect to the second argument of the function $f(\cdot,\cdot) $ only.
\end{lemma}


\begin{proof}[Proof of Lemma \ref{lemma: conv_alpha}] First of all, Lemma \ref{lemma: karimi_PL_conv}  implies that
\begin{align}
g(\theta_t)-f(\theta_t, \alpha_K(\theta_t)) \leq \rho^K\Delta.
\end{align}
Thus, using the QG result of Lemma \ref{lemma: QG}, we know that there exists an $\alpha^*\in A(\theta_t)$ such that
\begin{align}
\|\alpha_K(\theta_t) - \alpha^*\| \leq \rho^{K/2}\sqrt{\frac{\Delta}{2\mu}}
\end{align} 
Thus, 
\begin{align}
\|\nabla_\theta f(\theta_t, \alpha_K(\theta_t)) - \underbrace{\nabla_\theta f(\theta_t, \alpha^*)}_{\nabla g(\theta_t)}\|&\leq    L_{12}\|\alpha_K(\theta_t)-\alpha^*\|\nonumber\\
&\leq L_{12}\rho^{K/2}\sqrt{\frac{\Delta}{2\mu}}\leq \frac{\varepsilon}{4},
\end{align}
where the last inequality is due to the fact that $\frac{L_{12}^2\Delta}{\mu}\rho^K\leq \varepsilon^2/16$.
To prove the second part, note that 
\begin{align}
\|\nabla_{\alpha} f(\theta_t, \alpha_K(\theta_t)) - \underbrace{\nabla_\alpha f(\theta_t,\alpha^*(\theta_t))}_{0}\|&\leq L_{22}\|\alpha_K(\theta_t)- \alpha^*\|\nonumber\\
&\leq L_{22}\rho^{K/2}\sqrt{\frac{\Delta}{2\mu}}\leq \varepsilon,
\end{align}
where the last inequality is due to the fact that $\frac{L_{22}^2\Delta}{\mu}\rho^K\leq \varepsilon^2/16$.
\end{proof}
The above lemma implies that $\|\nabla_\theta f(\theta_t,\alpha_K(\theta_t))-\nabla g(\theta_t)\|\leq \delta = \frac{\varepsilon}{4}$.  We can use this result to show the convergence of our proposed algorithm to an $\varepsilon$--stationary solution of our min-max game. In other words, we show that  using $\nabla_\theta f(\theta_t,\alpha_K(\theta_t))$ instead of $\nabla g(\theta_t)$ in the gradient descent algorithm applied to $g$, leads to an $\varepsilon$-stationary solution. 
\begin{lemma}\label{lemma: conv_theta}
Given an $\varepsilon>0$, assume that we apply approximate gradient descent on an $L$-smooth function $g(\theta)$ starting from $\theta_0$, where $g(\theta_0)-g^*\leq \Delta_g$. In other words, at each iteration $t$, we perform the update rule $\alpha_{t+1} = \alpha_t -\frac{1}{L} G_t$ where $\|G_t-\nabla g(\theta_t)\|\leq \delta = \varepsilon/4$. Then, after $T = \frac{18L\Delta_g}{\varepsilon^2}$ iterations, there is at least one iteration $t \in \{0,1,\cdots,T-1\}$ such that $\|\nabla g(\theta_{t})\|\leq \frac{5\varepsilon}{12}$.
\end{lemma}

\begin{proof}
Based on the Taylor expansion of $g$, we have
\begin{align}
g(\theta_{t+1})\leq g(\theta_t) -\frac{1}{L}\langle\nabla g(\theta_t), G_t\rangle + \frac{1}{2L}\|G_t\|^2
\end{align}
Since $G_t = \underbrace{G_t - \nabla g(\theta_t)}_{e_t} + \nabla g(\theta_t)$,
\begin{align}
g(\theta_{t+1})&\leq g(\theta_t)-\frac{1}{2L}\|\nabla g(\theta_t)\|^2 + \frac{1}{2L}\|e_t\|^2 \nonumber\\
&\leq g(\theta_t)-\frac{1}{2L}\|\nabla g(\theta_t)\|^2 + \frac{1}{2L}\delta^2.
\end{align}
Summing up this inequality for all values of $t$ leads to
\begin{align}
\frac{1}{T}\sum_{t=0}^{T-1}\|\nabla g(\theta_t)\|^2 &\leq \delta^2 + \frac{2L(g(\theta_0)-g(\theta^T))}{T}\nonumber\\
&\leq \frac{\varepsilon^2}{16} + \frac{2L\Delta_g}{T}\leq \frac{25}{144}\varepsilon^2.
\end{align}
Therefore, the size of at least one of the gradients has to be less than $\frac{5}{12}\varepsilon$.
\end{proof}
Note that based on the above Lemma we know that in Algorithm \ref{alg: alg_grad} if we choose $K$ as in Lemma \ref{lemma: conv_alpha} and choose $T$ as in Lemma \ref{lemma: conv_theta} we get that at least for one $t\in\{0,\cdots,T-1\}$, we have that
\begin{align}
\|\nabla g(\theta_t)\|\leq \frac{5}{12}\varepsilon\quad{\rm and }\quad \|\nabla_\alpha f(\theta_t, \alpha_K(\theta_t))\|\leq \varepsilon.
\end{align}
Moreover, as we proved in Lemma~\ref{lemma: stability}, there exists an $\alpha^*\in A(\theta_t)$ such that $\|\alpha_K(\theta_t)-\alpha^*\|\leq \rho^{K/2}\sqrt{\frac{\Delta}{2\mu}}$. Thus,
\begin{align}
\|\nabla_\theta f(\theta_t,\alpha_K(\theta_t))-\underbrace{\nabla g(\theta_t)}_{\nabla_\theta f(\theta_t,\alpha^*)}\|\leq L_{22}\|\alpha_K(\theta_t)-\alpha^*\|\leq \frac{\varepsilon}{4},
\end{align}
where the last inequality is due to the choice of $K$ in Lemma \ref{lemma: conv_theta}. Now due to the fact that $\|\nabla g(\theta_t)\|\leq \frac{5}{12}\varepsilon$, we get
\begin{align}
\|\nabla_\theta f(\theta_t,\alpha_K(\theta_t))|\leq \varepsilon \quad{\rm and }\quad\|\nabla_\alpha f(\theta_t, \alpha_K(\theta_t))\|\leq \varepsilon,
\end{align}
which completes the proof of Theorem~\ref{thm:main}.
\end{proof}

\end{document}